\documentclass[notitlepage,11pt,reqno]{amsart}
\usepackage{amssymb,amsmath,color,tocvsec2}
\usepackage[mathscr]{eucal}
\usepackage[breaklinks=true]{hyperref}

\usepackage{pgf,tikz}
\usetikzlibrary{arrows}
\usepackage[letterpaper]{geometry}
\newcommand{\bn}{\ensuremath{\mathbf n}}
\newcommand{\R}{\ensuremath{\mathbb R}}
\newcommand{\D}{\ensuremath{\mathcal{D}}}
\newcommand{\W}{\ensuremath{\mathcal{W}}}

\DeclareMathOperator{\Id}{Id}
\DeclareMathOperator{\diag}{diag}

\newtheorem{theorem}[equation]{Theorem}
\newtheorem{corollary}[equation]{Corollary}
\newtheorem{proposition}[equation]{Proposition}

\theoremstyle{definition}
\newtheorem{definition}[equation]{Definition}

\numberwithin{equation}{section}

\begin{document}

\title{Multivariate blowup-polynomials of graphs}

\author{Projesh Nath Choudhury}
\address[P.N.~Choudhury]{Indian Institute of Science, Bangalore 560012,
India}
\email{\tt projeshc@iisc.ac.in}

\author{Apoorva Khare}
\address[A.~Khare]{Indian Institute of Science;
Analysis and Probability Research Group; Bangalore 560012, India}
\email{\tt khare@iisc.ac.in}

\date{\today}



\begin{abstract}
In recent joint work (2021), we introduced a novel multivariate
polynomial attached to every metric space -- in particular, to every
finite simple connected graph $G$ -- and showed it has several attractive
properties. First, it is multi-affine and real-stable (leading to a
hitherto unstudied delta-matroid for each graph $G$). Second, the
polynomial specializes to (a transform of) the characteristic polynomial
$\chi_{D_G}$ of the distance matrix $D_G$; as well as recovers the entire
graph, where $\chi_{D_G}$ cannot do so. Third, the polynomial encodes the
determinants of a family of graphs formed from $G$, called the blowups of
$G$.

In this short note, we exhibit the applicability of these tools and
techniques to other graph-matrices and their characteristic polynomials.
As a particular case, we will see that the adjacency characteristic
polynomial $\chi_{A_G}$ is in fact the shadow of a richer multivariate
blowup-polynomial, which is similarly multi-affine and real-stable.
Moreover, this polynomial encodes not only the aforementioned three
properties, but also yields additional information for specific families
of graphs.
\end{abstract}

\maketitle

\textit{Throughout this work, $G = (V,E)$ denotes a finite simple
connected graph (without self-loops or parallel edges).}

\section{Introduction}

This work provides novel connections between various matrices obtained
from graphs $G$, their spectra, and the geometry of real/complex
polynomials. It is a follow-up to our recent work~\cite{CK-blowup}, where
we were motivated by the problem of co-spectrality for (the
characteristic polynomial of) the distance matrix $D_G$. In that work, we
introduced a novel graph-invariant -- a multi-affine, real-stable
polynomial $p_G(\cdot)$ -- which
(a)~specializes to a transformation of the usual characteristic
polynomial of $D_G$, and
(b)~is able to recover the entire graph $G$ up to isometry, where the
univariate characteristic polynomial does not. Additionally,
(c)~$p_G(\cdot)$ encodes the determinant of $D_{G'}$ for $G'$ every
possible ``blowup'' of $G$:

\begin{definition}
Given a finite simple graph $G = (V,E)$, and a tuple $\bn = (n_v : v \in
V)$ of positive integers, the \emph{$\bn$-blowup} of $G$ is defined to be
the graph $G[\bn]$ -- with $n_v$ copies of each vertex $v$ -- such that a
copy of $v$ and one of $w$ are adjacent in $G[\bn]$ if and only if $v
\neq w$ are adjacent in $G$. (Blowups have previously been studied in the
context of extremal graph theory among other areas; see e.g.\
\cite{HHN,KSS}.)
\end{definition}

In this short note, we apply the tools and techniques developed
in~\cite{CK-blowup} to explore other well-known matrices in spectral
graph theory (and their corresponding characteristic polynomials). For
each graph, we will introduce a multivariate polynomial for each matrix
in a certain class, and prove that this polynomial has similarly
attractive properties as in~\cite{CK-blowup} (and mentioned above). As a
prototypical example, we will briefly discuss the adjacency matrix later
in this note.

We now introduce the matrices of interest. Given a graph $G = (V,E)$, we
study matrices of the form
\begin{equation}\label{Estart}
\mathcal{A}_G := \Delta_G + M_G,
\end{equation}
where:
\begin{itemize}
\item $M_G = (m_{vw})_{v,w \in V} \in \R^{V \times V}$ is a real
symmetric matrix (which encodes a graph-property) that is
``well-behaved'' under blowups, in that $M_{G[\bn]}$ is a block $V \times
V$ matrix, with the $(v,w)$ block equal to $m_{vw} {\bf 1}_{n_v \times
n_w}$.

\item $\Delta_G = \diag ( f_v )_{v \in V}$ is a nonsingular diagonal
matrix that is well-behaved under blowups, in that $\Delta_{G[\bn]}$ is a
$V \times V$ block diagonal matrix, with the $(v,v)$ block equal to $f_v
\Id_{n_v}$.
\end{itemize}

Here are a few examples of well-known matrices that are subsumed by this
paradigm:
\begin{enumerate}
\item $G$ is a finite simple connected graph, with $M_G = \D_G$ the
modified distance matrix studied in~\cite{CK-blowup}; and $\Delta_G = -2
\Id_V$. Thus $\mathcal{A}_G = \D_G - 2 \Id_V$ is precisely the `usual'
distance matrix $D_G = (d(v,w))_{v,w \in V}$, where $d(v,w)$ denotes the
(edge-)length of a shortest path joining $v,w$ in $G$. In fact the
framework for an arbitrary finite metric space $X$ is also subsumed by
the present model: $M_X = \D_X$ and $\Delta_X = \diag( -2 d_X(x, X
\setminus \{ x \}))_{x \in X}$, where $d_X(x, Y) := \min_{y \in Y}
d_X(x,y)$ for a non-empty subset $Y \subset X$.

\item $G$ is a finite simple graph, and $M_G = A_G$ is its adjacency
matrix. In this case, we fix a (nonzero) scalar $\lambda \in \R$ and let
$\Delta_G = \lambda \Id_V$.

\item $G$ is a finite simple graph, and $\mathcal{A}_G$ is the
\textit{Seidel matrix}, also studied in spectral graph theory. In this
case, $M_G = {\bf 1}_{V \times V} - 2 A_G$, and $\Delta_G = -\Id_V$.
\end{enumerate}

As in~\cite{CK-blowup}, we are motivated by the well-studied problem of
\textit{co-spectrality} with respect to a graph-matrix (e.g.\ the
distance matrix, or the adjacency/Seidel matrix). Recall that two graphs
$G \not\cong H$ are said to be co-spectral with respect to a graph-matrix
$M$ if the spectra of $M_G$ and $M_H$ agree as multi-sets (equivalently,
the characteristic polynomials of $M_G, M_H$ agree). Thus, a longstanding
problem in spectral graph theory is to understand, for a given
graph-matrix $M$, which pairs of non-isomorphic graphs are
$M$-cospectral.
In particular, it is well-known that all three matrices above admit such
graph pairs. In other words, none of these matrices $M_G$
\textit{detects} the underlying graph $G$ -- i.e., recovers the graph up
to isomorphism. See Figure~\ref{Fig1} for `small graph' examples for the
adjacency and Seidel matrices, and~\cite{DL} for an example for distance
matrices. (For completeness, we also refer the reader to the
texts~\cite{CDS,CRS} on spectral graph theory.)

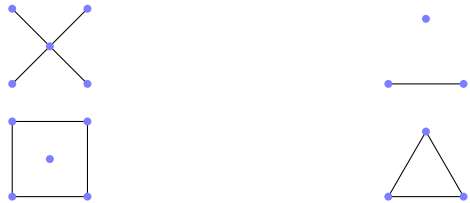
\begin{figure}[ht]
\definecolor{xdxdff}{rgb}{0.49,0.49,1}
\begin{tikzpicture}
\draw (0,0)-- (0,1);
\draw (0,0)-- (1,0);
\draw (1,0)-- (1,1);
\draw (0,1)-- (1,1);
\fill [color=xdxdff] (0,0) circle (1.5pt);
\fill [color=xdxdff] (0,1) circle (1.5pt);
\fill [color=xdxdff] (0.5,0.5) circle (1.5pt);
\fill [color=xdxdff] (1,0) circle (1.5pt);
\fill [color=xdxdff] (1,1) circle (1.5pt);
\draw (0,1.5)-- (0.5,2);
\draw (0,2.5)-- (0.5,2);
\draw (1,1.5)-- (0.5,2);
\draw (1,2.5)-- (0.5,2);
\fill [color=xdxdff] (0,1.5) circle (1.5pt);
\fill [color=xdxdff] (0,2.5) circle (1.5pt);
\fill [color=xdxdff] (0.5,2) circle (1.5pt);
\fill [color=xdxdff] (1,1.5) circle (1.5pt);
\fill [color=xdxdff] (1,2.5) circle (1.5pt);
\draw (5,0)-- (6,0);
\draw (5,0)-- (5.5,0.866);
\draw (6,0)-- (5.5,0.866);
\fill [color=xdxdff] (5,0) circle (1.5pt);
\fill [color=xdxdff] (6,0) circle (1.5pt);
\fill [color=xdxdff] (5.5,0.866) circle (1.5pt);
\draw (5,1.5)-- (6,1.5);
\fill [color=xdxdff] (5,1.5) circle (1.5pt);
\fill [color=xdxdff] (6,1.5) circle (1.5pt);
\fill [color=xdxdff] (5.5,2.366) circle (1.5pt);
\end{tikzpicture}
\caption{Two non-isomorphic graphs on five vertices that are adjacency
co-spectral; and two non-isomorphic graphs on three vertices that are
Seidel co-spectral}\label{Fig1}
\end{figure}

The purpose of this note is to show that in such cases, one can
nevertheless refine each such univariate polynomial -- in a natural
manner from multiple viewpoints: algebraic, spectral, and polynomial --
to obtain a multivariate real polynomial with several interesting
properties, listed presently. For instance, if $M_G$ is the (modified)
distance matrix or the adjacency matrix, then this polynomial does
recover the graph, hence is truly a graph-invariant. This yields a novel
family of multivariate graph-invariants for every graph $G = (V,E)$,
which we will term $p_{\mathcal{A}_G}(\cdot)$ -- see~\eqref{Estart}.
(This was carried out for the distance matrix in previous
work~\cite{CK-blowup}; we now provide a general model that works for all
graph-matrices $\mathcal{A}_G$ as in~\eqref{Estart}.) We also show that
in addition to being polynomials, these invariants
$p_{\mathcal{A}_G}(\cdot)$ have other attractive properties:
\begin{itemize}
\item They are multi-affine in their arguments $n_v, \ v \in V$.

\item They are \textit{real-stable} in the $n_v$.

\item The univariate specialization of $p_{\mathcal{A}_G}(\cdot)$ is
closely related to the characteristic polynomial of $M_G$, so that
$p_{\mathcal{A}_G}(\cdot)$ is indeed related to the spectrum of $M_G$.
Combined with the preceding point, this implies that the univariate
specialization is also real-rooted.

\item Unlike their univariate specializations (at least for the
distance/adjacency/Seidel matrices), the polynomials $p_{\mathcal{A}_G}$
recover the matrix $M_G^{\circ 2}$, whence $M_G$ if $M_G$ has
non-negative entries.

\item The polynomials $p_{\mathcal{A}_G}$ simultaneously encode the
determinants of the corresponding matrices $\mathcal{A}_{G[\bn]}$ for
\textit{all} graph-blowups (defined above) of $G$. Thus, in addition to
our original, `spectral' motivation, these polynomials also carry
algebraic information.
\end{itemize}

For the last-mentioned reason, we continue to adopt the notation
in~\cite{CK-blowup}, and call this object the
\textit{$\mathcal{A}$-blowup-polynomial of $G$}.
These polynomials are desirable in other ways as well. For example, it is
well-known that the (adjacency) spectrum of a bipartite graph $B = (V,E)$
is always symmetric about the origin, as a multi-set. We show in this
short note that the more general fact that the zero-locus of the
corresponding adjacency blowup-polynomial of $B$ is also symmetric around
the origin in $\mathbb{C}^V$.
Thus, the workings of~\cite{CK-blowup} (and now of this
paper) provide a broader recipe for a more refined study of
graph-polynomials.
We expect this line of investigation to lead to further multivariate
refinements of known results.

In a sense, this short note conforms to the philosophy that univariate
polynomials are special cases of multivariate ones, and these latter are
the more natural and general objects to study -- and they are more
powerful too. A famous recent manifestation of this has been in the
geometry of (the roots in $\mathbb{C}$ of) real and complex univariate
polynomials, where Borcea--Br\"and\'en and other mathematicians have
recently been extremely active in advancing the field (to cite a very few
sources, \cite{BB1,BB2,BB3,MSS1,MSS2,Wagner2}), a century after the
activity on the Laguerre--P\'olya--Schur program
\cite{Laguerre1,Polya1913,Polya-Schur}.
This recent progress has extensively advanced the theory of (real) stable
polynomials, with numerous applications including to negative dependence,
constructing bipartite Ramanujan graphs, and the Kadison--Singer problem.
Additionally, there are other several other examples. For instance,
Wagner's involved proof of the univariate Brown--Colbourn
conjecture~\cite{Wagner} was shortly followed by a one-paragraph proof of
its multivariate strengthening~\cite{Royle-Sokal,Sokal1}. Similarly,
multivariate versions \cite{Lieb-Sokal} of Lee--Yang type theorems for
Ising models, have been very influential in e.g.\ the Borcea--Br\"and\'en
program. We refer the reader to the excellent survey of
Sokal~\cite{Sokal2} for more instances and details.

\section{Algebraic results: polynomiality, coefficients, iterated
blowup, symmetries}

We now state and prove our results. The first set of assertions is
algebraic in nature.

\begin{theorem}\label{Tgeneral}
We retain the notation in~\eqref{Estart} and the lines immediately
thereafter. Thus $M_G = M_G^T$, and $f_v \in \R \setminus \{ 0 \}$ for
all $v \in V$.
\begin{enumerate}
\item There exists a polynomial $p_{\mathcal{A}_G} : \R^V \to \R$ such
that for all integer tuples $\bn \in \mathbb{Z}_{>0}^V$, we have:
\[
\det \mathcal{A}_{G[\bn]} = \prod_{v \in V} f_v^{n_v - 1} \cdot
p_{\mathcal{A}_G}(\bn), \qquad \forall \bn \in \mathbb{Z}_{>0}^V.
\]
In fact, the polynomial is (unique, and) given by:
\[
p_{\mathcal{A}_G}(\bn) = \det( \Delta_G + \Delta_\bn M_G), \qquad
\text{where} \qquad \Delta_\bn := \diag (n_v)_{v \in V}.
\]
In particular, $p_{\mathcal{A}_G}$ is multi-affine in $\bn$, with
constant and linear terms respectively equal to
\[
p_{\mathcal{A}_G}({\bf 0}) = \det \Delta_G = \prod_{v \in V} f_v \qquad
\text{and} \qquad 
(1,\dots,1) \cdot \Delta_\bn \cdot \nabla p_{\mathcal{A}_G}({\bf 0}) =
p_{\mathcal{A}_G}({\bf 0}) \sum_{v \in V} \frac{m_{vv}}{f_v} n_v.
\]

\item More generally, for each $I \subset V$ the coefficient in
$p_{\mathcal{A}_G}(\bn)$ of $\prod_{v \in I} n_v$ equals
\[
\det (M_G)_{I \times I} \prod_{v \in V \setminus I} f_v.
\]

\item The $\mathcal{A}$-blowup-polynomial of the blowup $G[\bn]$ has a
closed-form expression. More precisely, given an integer tuple $\bn \in
\mathbb{Z}_{>0}^V$ and variables $\{ m_{vi} : v \in V, \ 1 \leq i \leq
n_v \}$, define $n'_v := \sum_i m_{vi}$ and $\bn' := (n'_v)_{v \in V}$.
Then:
\[
p_{\mathcal{A}_{G[\bn]}}({\bf m}) \equiv 
p_{\mathcal{A}_G}({\bf n}') \prod_{v \in V} f_v^{n_v - 1}.
\]
\end{enumerate}
\end{theorem}

As an immediate consequence of the final part, the polynomials
$p_{\mathcal{A}_{G[\bn]}}(\cdot)$ all have total degree at most $|V|$,
regardless of the integer tuple $\bn$.

\begin{proof}
The second part follows easily from the first (see e.g.\ \cite[Section
2]{CK-blowup}). The first part was shown in \textit{loc.\ cit.}\ using
arguments from commutative algebra (specifically, Zariski density), and
these arguments also apply here, so that the result holds over an
arbitrary unital commutative ring. In the interest of variance, we now
provide an alternate, more direct proof using real numbers. Begin by
noticing that in the desired assertion
\[
\det \mathcal{A}_{G[\bn]} = \prod_{v \in V} f_v^{n_v - 1} 
\det( \Delta_G + \Delta_\bn M_G),
\]
both sides are polynomial functions, whence continuous, in the entries of
(the real symmetric matrix) $M_G$. In particular, if $M_G$ is singular,
we may replace it by $M_G + \epsilon \Id_V$ for small $\epsilon > 0$, and
let $\epsilon \to 0^+$. Thus, without loss of generality, we may assume
$\det M_G \neq 0$.

We now proceed to the proof. Define the integers $0 < k \leq K$ and the
matrix $\W$ via:
\[
k := |V|, \qquad K := \sum_{v \in V} n_v, \qquad
\W_{K \times k} := \begin{pmatrix} {\bf 1}_{n_1 \times 1} & 0_{n_1
\times 1} & \cdots & 0_{n_1 \times 1} \\
0_{n_2 \times 1} & {\bf 1}_{n_2 \times 1} & \cdots & 0_{n_2 \times 1} \\
\vdots & \vdots & \ddots & \vdots \\
0_{n_k \times 1} & 0_{n_k \times 1} & \cdots & {\bf 1}_{n_k \times 1}
\end{pmatrix},
\]
where $(n_1, \dots, n_k)$ is a fixed enumeration of the integers $n_v$.
Now compute, with a repeated use of Schur complements, and using that
$\Delta_G = \diag (f_v)_{v \in V}$ is invertible:
\begin{align*}
\det \mathcal{A}_{G[\bn]} = &\ \det (\Delta_{G[\bn]} + \W M_G \W^T) =
\det \begin{pmatrix} \Delta_{G[\bn]} & - \W \\ \W^T &
M_G^{-1} \end{pmatrix} \det(M_G) \\
= &\ \det (\Delta_{G[\bn]}) \det (M_G^{-1} + \W^T \Delta_{G[\bn]}^{-1}
\W) \det (M_G),
\end{align*}
where we label $\Delta_G, M_G$ compatibly with the enumeration
$(n_j)_{j=1}^k$. Now since $\W^T \Delta_{G[\bn]}^{-1} \W = \diag(n_v
f_v^{-1})_{v \in V}$, we continue:
\begin{align*}
= &\ \prod_{v \in V} f_v^{n_v} \cdot \det( \Id_V + \Delta_G^{-1}
\Delta_\bn M_G) = \prod_{v \in V} f_v^{n_v - 1} \cdot \det (\Delta_G +
\Delta_\bn M_G).
\end{align*}

This proves the first part, except for the uniqueness of the polynomial
-- but this follows from the Zariski density of $\mathbb{Z}_{>0}^V$ in
$\R^V$, which simply means that any polynomial on $\R^V$ that vanishes on
$\mathbb{Z}_{>0}^V$ is identically zero. (There is some more work to do
if one wants to prove this result over an arbitrary commutative ring, as
was done in \cite{CK-blowup}.)

It remains to show the third part. Once again, we avoid Zariski density
arguments as in~\cite{CK-blowup}, and work with the real symmetric matrix
$M_G$. As above, we may assume $M_G$ is invertible (as is $\Delta_G$). We
compute using the first part, and Schur complements:
\begin{align*}
p_{\mathcal{A}_{G[\bn]}}({\bf m}) = &\ \det (\Delta_{G[\bn]} +
\Delta_{\bf m} \W M_G \W^T)\\
= &\ \det(\Delta_{\bf m}) \det (\Delta_{\bf m}^{-1} \Delta_{G[\bn]} + \W
M_G \W^T)\\
= &\ \det(\Delta_{\bf m}) \det \begin{pmatrix} \Delta_{\bf m}^{-1}
\Delta_{G[\bn]} & - \W \\ \W^T & M_G^{-1} \end{pmatrix} \det(M_G)\\
= &\ \det (\Delta_{\bf m}) \det (\Delta^{-1}_{\bf m}) \det
(\Delta_{G[\bn]}) \det (M_G^{-1} + \W^T \Delta_{\bf m}
\Delta_{G[\bn]}^{-1} \W) \det(M_G).
\end{align*}

But $\W^T \Delta_{\bf m} \Delta_{G[\bn]}^{-1} \W = \diag (f_v^{-1}
n'_v)_{v \in V}$, where $n'_v := \sum_i m_{vi}$ as above. Hence we
continue:
\begin{align*}
= \prod_{v \in V} f_v^{n_v} \cdot \det( M_G^{-1} + \Delta_G^{-1}
\Delta_{\bn'}) \det (M_G) = \prod_{v \in V} f_v^{n_v -1} \cdot \det(
\Delta_G + \Delta_{\bn'} M_G),
\end{align*}
which proves the third part.
\end{proof}

A consequence of the preceding result is a linear delta-matroid that
arises from the $\mathcal{A}$-blowup-polynomial:

\begin{corollary}\label{Cdelta}
Setting as in Theorem~\ref{Tgeneral}. The set of monomials with nonzero
coefficients in $p_{\mathcal{A}_G}(\bn)$ forms the linear delta-matroid
$\mathcal{M}_{M_G}$.
\end{corollary}

Recall that delta-matroids were defined and studied by
Bouchet~\cite{Bouchet1}, and generalize the notion of a matroid. Given a
symmetric matrix $M$ over a field, Bouchet showed in~\cite{Bouchet2} that
the subsets of indices corresponding to non-vanishing principal minors,
form a (linear) delta-matroid $\mathcal{M}_M$. This explains how the
corollary follows immediately from Theorem~\ref{Tgeneral}(2).

The next result shows that the $\mathcal{A}$-blowup-polynomial
$p_{\mathcal{A}_G}$, together with the scalars $f_v, m_{vv}$, determine
the rest of the matrix $M_G$ -- or more precisely, its Hadamard square
$M_G^{\circ 2} = (m_{ij}^2)$.

\begin{proposition}\label{Precover}
The homogeneous quadratic part of $p_{\mathcal{A}_G}$, i.e.\ its Hessian
at the origin, equals
\[
\mathcal{H}(p_{\mathcal{A}_G}) := ((\partial_{n_v} \partial_{n_{v'}}
p_{\mathcal{A}_G}))_{v,v' \in V} |_{\bn = {\bf 0}_V} = \prod_{v \in V}
f_v \cdot \Delta_G^{-1} ({\bf m}_G {\bf m}_G^T - M_G^{\circ 2})
\Delta_G^{-1},
\]
where ${\bf m}_G := (m_{vv})_{v \in V} \in \R^V$ is the (column) vector
containing the diagonal entries of $M_G$. Thus if the scalars $f_v,
m_{vv}$ are known, then the $\mathcal{A}$-blowup-polynomial
$p_{\mathcal{A}_G}(\bn)$ detects the matrix $M_G^{\circ 2}$.
\end{proposition}

For instance, if one works with the distance or adjacency matrix of $G$,
then all entries in $M_G$ are non-negative, and the entries $f_v, m_{vv}$
are also known (see above), so Proposition~\ref{Precover} recovers the
entire matrix $M_G$, and hence the graph $G$.

As the proof of Proposition~\ref{Precover} is based on a direct
computation using Theorem~\ref{Tgeneral}(2), and is similar to the
corresponding proof in our recent paper~\cite{CK-blowup}, we omit it
here. (Note however that the formula here is more general than its
counterpart in~\cite{CK-blowup}.)

\section{Results on real-stability}

We now move to results connecting the blowup-polynomial with spectral
graph theory and the geometry of (real) polynomials. The next result
provides a sufficient condition for the real-stability of
$p_{\mathcal{A}_G}$ in the above paradigm (and hence the real-rootedness
of its univariate specialization):

\begin{theorem}\label{Tstable}
Setting as in Theorem~\ref{Tgeneral}. Define $u_{\mathcal{A}_G}(x) :=
p_{\mathcal{A}_G}(x,x,\dots,x)$.
\begin{enumerate}
\item Suppose all scalars $f_v$ are (nonzero and) of the same sign. Then
$p_{\mathcal{A}_G}(\cdot)$ is real-stable. (This means that if all
arguments $\bn$ lie in the open upper half-plane $\Im(z) > 0$, then
$p_{\mathcal{A}_G}(\bn) \neq 0$.)
In particular, $u_{\mathcal{A}_G}$ is real-rooted.

\item Suppose in fact that all scalars $f_v$ are equal, say to $\lambda
\neq 0$. Then $x \in \R$ is a root of $u_{\mathcal{A}_G}$ if and only if
$x \neq 0$ and $-\lambda/x$ is an eigenvalue of $M_G$.
\end{enumerate}
\end{theorem}

We remark here that these hypotheses are indeed satisfied when one
studies distance matrices of graphs \cite{CK-blowup}, or the adjacency or
Seidel matrices as above.

\begin{proof}\hfill
\begin{enumerate}
\item Suppose $\varepsilon \in \{ \pm 1 \}$ is the sign of every $f_v$,
so that the diagonal matrix $\varepsilon \Delta_G$ is positive definite.
We compute, allowing for the $n_v$ to now be \textit{complex} variables:
\[
p_{\mathcal{A}_G}(\bn) = \det(\Delta_\bn) (-\varepsilon)^{|V|} \det(
-\varepsilon M_G - \Delta_\bn^{-1} (\varepsilon \Delta_G)).
\]
Let $E_{vv}$ denote the elementary $V \times V$ matrix, with $1$ in the
$(v,v)$ entry and all other entries zero. Then,
\begin{equation}\label{Etemp}
p_{\mathcal{A}_G}(\bn) = \det (\Delta_\bn) (-\varepsilon)^{|V|} \det
\left( - \varepsilon M_G + \sum_{v \in V} (-n_v^{-1}) (\varepsilon f_v)
E_{vv} \right).
\end{equation}

Ignoring the scalar $(-\varepsilon)^{|V|}$, it suffices to show that the
determinant-expression times $\det (\Delta_\bn)$ is real-stable in the
$n_v$. Here, we recall a result of Borcea and Br\"and\'en
\cite[Proposition 2.4]{BB1}, which says that the determinantal polynomial
$\displaystyle \det (B + \sum_{j=1}^m z_j A_j)$
is real-stable in the $z_j$ if all $A_j$ are positive semidefinite and
$B$ is real symmetric. 
As an application, since the matrices $\varepsilon f_v E_{vv}$ are
positive semidefinite for all $v \in V$, the determinant in~\eqref{Etemp}
(without the extra factor of $\prod_v n_v$) is real-stable, provided that
one replaces each $(-n_v^{-1})$ by $n_v$. Now use that stability is
preserved under `inversion': if a polynomial $p(\{ n_v \})$ with
$n_w$-degree $d$ is stable (for some fixed $w \in V$), then so  is $n_w^d
p( \{ n_v : v \neq w \}, - n_w^{-1})$.
Applying this for each variable $n_v$ in turn, the first assertion
follows.

\item By Theorem~\ref{Tgeneral}(1),
$u_{\mathcal{A}_G}(x) = \det (\lambda \Id_V + x M_G)$,
and this does not vanish if $x=0$, since $\lambda \neq 0$. Thus, $x$
is a root here, if and only if $x \neq 0$ and
\[
0 = x^{-|V|} u_{\mathcal{A}_G}(x) = \det (\lambda x^{-1} \Id_V + M_G),
\]
and the result is immediate from here.\qedhere
\end{enumerate}
\end{proof}

The next result provides necessary conditions and sufficient conditions
for when a graph is a blowup, in terms of the matrix $\mathcal{A}_G$:

\begin{proposition}\label{Pblowup}
Notation as above; also suppose that all scalars $f_v$ are equal, say to
$\lambda \in \mathbb{R}$. Then each of the following statements implies
the next:
\begin{enumerate}
\item $G$ is a nontrivial blowup. In other words, $G$ is a blowup of a
graph $H$ with $|V(H)| < |V(G)|$.

\item There exist two vertices $v \neq w$ in $G$ which share the same set
of neighbors. (Thus, $v,w$ are not adjacent.)

\item $\lambda$ is an eigenvalue of the matrix $\mathcal{A}_G$.

\item The blowup-polynomial $p_{\mathcal{A}_G}$ has total degree at most
$|V|-1$.
\end{enumerate}
In fact the first two assertions are equivalent (and do not depend on
$\mathcal{A}_G$), and so are the last two.
\end{proposition}

\begin{proof}
The first two assertions are taken from~\cite{CK-blowup}; we reproduce
that proof here. If $(1)$ holds then there are two distinct vertices
which are copies of one another, proving $(2)$. Conversely, if $(2)$
holds then $G$ is a blowup of the smaller graph where one of these two
vertices is removed, proving $(1)$.

We next show $(2) \implies (3)$: if $(2)$ holds, then $M_G$ contains two
identical rows, hence is singular. But then $\mathcal{A}_G = M_G +
\lambda \Id_V$ has $\lambda$ as an eigenvalue, proving $(3)$. Finally,
$(3)$ holds if and only if $M_G$ is singular, if and only if (by
Theorem~\ref{Tgeneral}(2)) the unique monomial in
$p_{\mathcal{A}_G}(\bn)$ of top degree is zero, proving that $(3)
\Longleftrightarrow (4)$.
\end{proof}

Our final result in this section takes an in-depth look at the
real-stability of the polynomial $p_{\mathcal{A}_G}(\bn)$. Notice by
Theorem~\ref{Tgeneral}(1) that these polynomials are not homogeneous; nor
are their coefficients \textit{a priori} all of the same sign.
Real-stable polynomials with either of these two properties are the
subjects of active research: the former fall in the broader family of
\textit{Lorentzian polynomials}~\cite{BH}, as well as
\textit{strongly/completely log-concave polynomials
\cite{AGV,Gurvits}}; and the latter are termed \textit{strongly Rayleigh
polynomials}, and are important in probability and the theory of negative
dependence (see~\cite{BBL} and the references therein). We first provide
the necessary definitions.

\begin{definition}
Fix a real polynomial of complex variables $p(z_1, \dots, z_k) \in
\R[z_1, \dots, z_k]$.
\begin{enumerate}
\item A real symmetric matrix is \emph{Lorentzian} if it is nonsingular
with only one positive eigenvalue.

\item Following~\cite{BH}, we say $p$ is \emph{Lorentzian} if $p$ is
homogeneous of some degree $d \geq 2$, has non-negative coefficients, and
given indices $1 \leq j_1, \dots, j_{d-2} \leq k$, if we define
\[
g(z_1, \dots, z_k) := \left( \partial_{z_{j_1}} \cdots
\partial_{z_{j_{d-2}} } p \right)(z_1, \dots, z_k),
\]
then its Hessian matrix $\mathcal{H}_g := (\partial_{z_i} \partial_{z_j}
g)_{i,j=1}^k \in \R^{k \times k}$ is Lorentzian.

\item Following~\cite{Gurvits}, we say that $p$ is \emph{strongly
log-concave} if $p$ has all coefficients non-negative, and for all tuples
$\alpha \in \mathbb{Z}_{\geq 0}^k$, either the derivative $\displaystyle
\partial^\alpha (p) := \prod_{i=1}^k \partial_{z_i}^{\alpha_i} \cdot p$
is identically zero, or $\log(\partial^\alpha (p))$ is defined and
concave on $(0,\infty)^k$.

\item Following~\cite{AGV}, we say that $p$ is \emph{completely
log-concave} if $p$ has all coefficients non-negative, and for all
integers $m \geq 0$ and matrices $A = (a_{ij}) \in [0,\infty)^{m \times
k}$, either the derivative $\displaystyle \partial_A (p) := \prod_{i=1}^m
\left( \sum_{j=1}^k a_{ij} \partial_{z_j} \right) \cdot p$ is identically
zero, or $\log(\partial_A (p))$ is defined and concave on $(0,\infty)^k$.

\item We say $p$ is \emph{strongly Rayleigh} if $p$ is multi-affine and
real-stable in the $z_j$, and has all coefficients non-negative and
summing to $1$.
\end{enumerate}
\end{definition}

We now explore when the homogenized blowup-polynomial of
$p_{\mathcal{A}_G}$ is Lorentzian -- or its normalization is strongly
Rayleigh -- in the spirit of a result proved in~\cite{CK-blowup} for
distance matrices. The following adapts that result to the current
setting.

\begin{theorem}\label{Tlorentz}
Setting as in Theorem~\ref{Tgeneral}, with $k := |V| \geq 2$. Also
suppose that all scalars $f_v \in \R$ are nonzero and have the same sign
$\varepsilon \in \{ \pm 1 \}$. Define the homogenized polynomial
\[
\widetilde{p}_{\mathcal{A}_G}(z_0, z_1, \dots, z_k) := (\varepsilon
z_0)^k p_{\mathcal{A}_G}(\varepsilon z_1 / z_0, \dots, \varepsilon z_k /
z_0) \in \R[z_0, z_1, \dots, z_k],
\]
with possibly complex arguments.
Then the following statements are equivalent.
\begin{enumerate}
\item $\widetilde{p}_{\mathcal{A}_G}(z_0, z_1, \dots, z_k)$ is
real-stable.

\item $\widetilde{p}_{\mathcal{A}_G}(z_0, z_1, \dots, z_k)$ is
Lorentzian (equivalently, strongly / completely log-concave).

\item All coefficients of $\widetilde{p}_{\mathcal{A}_G}(z_0, z_1, \dots,
z_k)$ are non-negative.

\item We have $\varepsilon^k p_{\mathcal{A}_G}(\varepsilon, \dots,
\varepsilon) > 0$, and the following polynomial is strongly Rayleigh:
\[
(z_1, \dots, z_k) \quad \mapsto \quad \frac{p_{\mathcal{A}_G}(\varepsilon
z_1, \dots, \varepsilon z_k)}{p_{\mathcal{A}_G}(\varepsilon, \dots,
\varepsilon)}.
\]

\item The matrix $M_G$ is positive semidefinite.
\end{enumerate}
\end{theorem}

Before proving Theorem~\ref{Tlorentz}, we note that it is a `negative'
result for the graph-properties discussed in this paper. For example, for
the distance matrix the only graphs for which $M_G = D_G + 2 \Id_V$ is
positive semidefinite, are complete (multipartite) graphs,
by~\cite{LHWS}. For the adjacency matrix, the only graphs for which $M_G$
is positive semidefinite are the graphs with no edges. Nevertheless, the
family of matrices $\mathcal{A}_G$ as in~\eqref{Estart} can contain other
examples for which the matrix $M_G$ is positive semidefinite.

\begin{proof}[Proof of Theorem~\ref{Tlorentz}]
That $(1) \implies (2)$ was shown in~\cite{BH}, and that a Lorentzian
polynomial satisfies~(3) follows from the definitions. The equivalences
inside assertion~(2) were shown in \cite[Theorem 2.30]{BH}.
We next show that $(1) \implies (4) \implies (3)$. Given $(1)$ (and hence
$(3)$), we see that the sum of all coefficients in
$\widetilde{p}_{\mathcal{A}_G}$ equals its value at $(1,1,\dots,1)$, and
so using $(3)$:
\[
\varepsilon^k p_{\mathcal{A}_G}(\varepsilon,\dots,\varepsilon) =
\widetilde{p}_{\mathcal{A}_G}(1,1,\dots,1) \geq
\widetilde{p}_{\mathcal{A}_G}(1,0,\dots,0) = \prod_{v \in V} (\varepsilon
f_v) > 0.
\]
Moreover, the coefficients of the normalized polynomial
\begin{equation}\label{Erayleigh}
\frac{p_{\mathcal{A}_G}(\varepsilon z_1, \dots, \varepsilon
z_k)}{p_{\mathcal{A}_G}(\varepsilon, \dots, \varepsilon)} =
\frac{\widetilde{p}_{\mathcal{A}_G}(1, z_1, \dots,
z_k)}{\widetilde{p}_{\mathcal{A}_G}(1, 1, \dots, 1)}
\end{equation}
are non-negative and add up to one. Finally, the left-hand side is
real-stable because the right-hand side is, by~$(1)$ and by specializing
at $z_0 \mapsto 1$ (which preserves real-stability). This shows~$(4)$.
Conversely, if~$(4)$ holds, then Equation~\eqref{Erayleigh}
implies~$(3)$, as desired.

Finally, if~$(3)$ holds, then Theorem~\ref{Tgeneral}(2) (and the
hypotheses that $\varepsilon f_v > 0\ \forall v \in V$) implies that
every principal minor of $M_G$ is non-negative. This is because the
coefficient of $z_0^{k - |J|} \prod_{v \in J} z_v$ equals $\prod_{v
\in V \setminus J} (\varepsilon f_v) \cdot \det (M_G)_{J \times J}$, and
this is to be non-negative for every subset $J$. This shows~$(5)$.
Conversely, if~$(5)$ holds, then we use the positive semidefinite matrix
\[
\mathcal{C}_G := (\varepsilon \Delta_G)^{-1/2} M_G (\varepsilon
\Delta_G)^{-1/2}
\]
in the following computation:
\begin{align*}
\widetilde{p}_{\mathcal{A}_G}(z_0, z_1, \dots, z_k) = &\ \det
(\varepsilon \Delta_G)^{1/2} \det (z_0 \Id_k + \Delta_{\bf z}
\mathcal{C}_G) \det (\varepsilon \Delta_G)^{1/2}\\
= &\ \det (\varepsilon \Delta_G) \det( z_0 \Id_k + \sqrt{\mathcal{C}_G}
\Delta_{\bf z} \sqrt{\mathcal{C}_G})
= \det ( z_0 \Id_k + \sum_{v \in V} z_v \sqrt{\mathcal{C}_G} E_{vv}
\sqrt{\mathcal{C}_G} ),
\end{align*}
where the second equality comes from expanding
$\det \begin{pmatrix} z_0 \Id_k & -\sqrt{\mathcal{C}_G} \\ \Delta_{\bf z}
\sqrt{\mathcal{C}_G} & \Id_k \end{pmatrix}$
in two ways, both via Schur complements; and where the matrix $E_{vv}$
was defined prior to~\eqref{Etemp}. Now the final expression is indeed
real-stable by \cite[Proposition 2.4]{BB1} (see the discussion
following~\eqref{Etemp}), so~$(1)$ holds.
\end{proof}

\section{The adjacency blowup-polynomial}

We conclude this note by illustrating several of the above results in the
general case, by specializing them to the adjacency matrix of a graph and
its blowup-polynomial, as a particular example.

Suppose $M_G = A_G$ is the adjacency matrix (so $m_{vv} = 0\ \forall v
\in V$), and $\lambda \in \R$ is any fixed nonzero scalar. Defining
$\mathcal{A}_{G,\lambda} = \lambda \Id_V + A_G$, we obtain a real-stable
polynomial $p_{\mathcal{A}_{G,\lambda}}(\bn)$. Now the coefficient in
$p_{\mathcal{A}_{G,\lambda}}(\bn)$ of any monomial $\bn^I$ (where $I
\subset V$) is a scalar times the determinant of $(A_G)_{I \times I}$;
note this principal submatrix is itself block diagonal, with components
corresponding to the connected components of the induced subgraph on $I$
(this makes $\det (A_G)_{I \times I}$ `easier' to compute). This
observation and others lead to the following result.

\begin{proposition}
Suppose $M_G = A_G$, the adjacency matrix of a graph $G$; and $\Delta_G =
\lambda \Id_V$ for a fixed nonzero scalar $\lambda \in \R$.
\begin{enumerate}
\item Suppose $H \subset G$ is an induced subgraph. Then
\[
p_{\mathcal{A}_{H,\lambda}}(\{ n_v : v \in V(H) \}) =
p_{\mathcal{A}_{G,\lambda}}(\bn) |_{n_v = 0\ \forall v \in V(G) \setminus
V(H)} \cdot \lambda^{|V(H)| - |V(G)|}.
\]

\item If $H,H'$ are isomorphic subgraphs inside $G$, then the
coefficients in $p_{\mathcal{A}_{G,\lambda}}$ of the monomials
corresponding to $V(H)$ and $V(H')$ are equal.

\item Suppose for some non-empty subset $J \subset V(G)$ that the induced
subgraph on $J$ contains a connected component which is a tree without a
perfect matching. Then the coefficient of $\bn^J$ in
$p_{\mathcal{A}_{G,\lambda}}(\bn)$ is zero.

\item If $\lambda = -1$, then up to a scalar, the univariate
specialization $u_{\mathcal{A}_{G,-1}}(x)$ is precisely the `inversion'
of the characteristic polynomial of $A_G$:
$u_{\mathcal{A}_{G,-1}}(z) = z^{|V|} \chi_{A_G}(z^{-1})$.
\end{enumerate}
\end{proposition}

The final part allows one to interpret the eigenvalues of $A_G$ in terms
of the roots of $u_{\mathcal{A}_{G,-1}}$. (E.g.\ the second largest
eigenvalue is important for studying $d$-regular bipartite Ramanujan
graphs.)

\begin{proof}
The key fact used in these results is that the adjacency matrix of any
subgraph on $J \subset V(G)$ is the principal $J \times J$ submatrix of
$A_G$. This fact, combined with Theorem~\ref{Tgeneral}(2) and that
$\Delta_G$ is a scalar matrix, immediately yields the second part.
The `key fact' also yields the first part via setting all other variables
$n_v, v \not\in V(H)$ to zero, since this yields a matrix with only the
diagonal entry nonzero in each row not indexed by $V(H)$. The first part
follows easily from here.

The third part holds because of the observation immediately preceding
this proposition, combined with the classical fact (see e.g.~\cite{CDS})
that the adjacency matrix of a tree is nonsingular if and only if the
tree has a perfect matching. The final part follows immediately from
Theorem~\ref{Tstable}(2).
\end{proof}

The multivariate blowup-polynomial also has other attractive properties;
we mention one that is crucially used in studying bipartite graphs
(including in constructing bipartite Ramanujan expanders in~\cite{MSS1}
and its precursors). A folklore result is that $G$ is bipartite if and
only if its adjacency spectrum is symmetric (as a multiset) around the
origin. In fact, this extends to the zero-locus of the blowup-polynomial
-- now yielding an even polynomial since $u_{\mathcal{A}_{G,-1}}$ is the
`inversion' of the adjacency characteristic polynomial:

\begin{proposition}
Suppose $G$ is a graph, $\lambda \in \R$ is nonzero, and we set
$\mathcal{A}_{G,\lambda} = \lambda \Id_V + A_G$. The following are
equivalent:
\begin{enumerate}
\item $p_{\mathcal{A}_{G,\lambda}}$ is even in $\bn$.

\item $G$ is bipartite.
\end{enumerate}
\end{proposition}

In particular, the zeros of $p_{\mathcal{A}_{G,\lambda}}$ for any
bipartite graph $G$ are symmetric around the origin, and there are no
odd-degree monomials in $p_{\mathcal{A}_{G,\lambda}}$.

\begin{proof}
First suppose $G$ is bipartite. Write the adjacency matrix as $A_G =
\begin{pmatrix} 0 & B^T \\ B & 0 \end{pmatrix}$, and compute:
\begin{align*}
p_{\mathcal{A}_{G,\lambda}}(-\bn) = &\ \det \begin{pmatrix} - \Id & 0 \\
0 & \Id \end{pmatrix} \det (\lambda \Id_V - \Delta_\bn A_G) \det
\begin{pmatrix} - \Id & 0 \\ 0 & \Id \end{pmatrix}\\
= &\ \det \left( \lambda \Id_V - \Delta_\bn
\begin{pmatrix} - \Id & 0 \\ 0 & \Id \end{pmatrix} 
\begin{pmatrix} 0 & B^T \\ B & 0 \end{pmatrix}
\begin{pmatrix} - \Id & 0 \\ 0 & \Id \end{pmatrix} \right)\\
= &\ \det (\lambda \Id_V + \Delta_\bn A_G) =
p_{\mathcal{A}_{G,\lambda}}(\bn),
\end{align*}
as desired.

Conversely, suppose~(1)
$p_{\mathcal{A}_{G,\lambda}}(\bn) \equiv
p_{\mathcal{A}_{G,\lambda}}(-\bn)$.
Now it is well-known that the adjacency matrix $A_{C_m}$ for any cycle
graph $C_m$ on $m$ vertices, is circulant, and has eigenvalues $2 \cos(2
\pi j / m)$ for $0 \leq j < m$. If $m$ is odd then no eigenvalue is zero
and $A_{C_m}$ is non-singular.
Returning to the proof of the converse, if $G$ is not bipartite then it
has an induced odd cycle, say with vertices $v_1, \dots, v_m$ for $m$
odd. In particular, the coefficient of $n_{v_1} \cdots n_{v_m}$ in
$p_{\mathcal{A}_{G,\lambda}}$ is nonzero, by Theorem~\ref{Tgeneral}(2).
But this contradicts~(1). The converse follows.
\end{proof}

We conclude with two directions of future investigation. In one
direction, as we just saw, for particular classes of graphs (e.g.,
bipartite graphs), one can obtain additional information that refines and
enriches the univariate picture. Another direction involves the adjacency
blowup-polynomial and blowup delta-matroid, and exploring their
connections to previously studied notions and invariants arising from
combinatorics. For example, Bouchet showed in~\cite{Bouchet3} that for
any graph $G = (V,E)$, the subsets $I \subset V$ for which the induced
subgraph on $I$ has a perfect matching, comprise a delta-matroid. Now by
Corollary~\ref{Cdelta}, the monomials with nonzero coefficients in
$p_{\mathcal{A}_{G,\lambda}}$ form the adjacency-blowup delta-matroid of
$G$ (independent of $\lambda \neq 0$). Notice that these constructions
agree whenever $G$ is a tree. It is thus natural to ask if these two
constructions agree in general; but even for small graphs $G$, this is
not the case. For example, for $G = C_5$, the subset $V(G)$ cannot have a
perfect matching, yet occurs in the blowup delta-matroid
$\mathcal{M}_{\mathcal{A}_{C_5,\lambda}}$
(defined in the paragraph after Corollary~\ref{Cdelta});
and for $G = C_4$, the `reverse' holds: $V(G)$ has a perfect matching,
but does not occur in $\mathcal{M}_{\mathcal{A}_{C_4,\lambda}}$ (e.g.\ by
Proposition~\ref{Pblowup}, since $C_4$ is a blowup).

\subsection*{Acknowledgments}
P.N.~Choudhury was supported by National Post-Doctoral Fellowship (NPDF)
PDF/2019/000275 from SERB (Govt.~of India).
A.~Khare was partially supported by
Ramanujan Fellowship grant SB/S2/RJN-121/2017,
MATRICS grant MTR/2017/000295, and
SwarnaJayanti Fellowship grants SB/SJF/2019-20/14 and DST/SJF/MS/2019/3
from SERB and DST (Govt.~of India),
and by grant F.510/25/CAS-II/2018(SAP-I) from UGC (Govt.~of India).




\end{document}